\documentclass[nointlimits,11pt,oneside]{amsart}
\usepackage[numbers]{natbib}
\usepackage{amssymb,cases,enumitem}
\usepackage{color}
\usepackage{hyperref}

\usepackage[%
	a4paper,
	total={16cm,23cm},
	left=3cm, top=3cm,
	marginparsep=2pt]
{geometry}

\theoremstyle{plain}
\newtheorem{theorem}{Theorem}[section]
\newtheorem{lemma}[theorem]{Lemma}
\newtheorem{corollary}[theorem]{Corollary}
\newtheorem{proposition}[theorem]{Proposition}
\theoremstyle{definition}

\newtheorem{definition}[theorem]{Definition}

\numberwithin{equation}{section}

\def\L1loc{L^1_{\text{loc}}}

\begin{document}

\title{On the properties of quasi-Banach function spaces}
\author{Ale\v s Nekvinda and Dalimil Pe\v sa}

\address{Ale\v s Nekvinda, Department of Mathematics, Faculty of Civil Engineereng,
	Czech Technical University, Th\' akurova 7, 16629 Prague~6,
	Czech Republic}
\email{ales.nekvinda@cvut.cz}
\urladdr{0000-0001-6303-5882}

\address{Dalimil Pe{\v s}a, Department of Mathematical Analysis, Faculty of Mathematics and
	Physics, Charles University, Sokolovsk\'a~83,
	186~75 Praha~8, Czech Republic}
\email{pesa@karlin.mff.cuni.cz}
\urladdr{0000-0001-6638-0913}


\thanks{The research of both authors was supported by the grant no. P201-18-00580S of the Grant Agency of the Czech Republic; the second author was further supported by the grant SVV-2020-260583.}

\begin{abstract}
In this paper we explore some basic properties of quasi-Banach function spaces which are important in applications. Namely, we show that they possess a generalised version of Riesz--Fischer property, that embeddings between them are always continuous and that the dilation operator is bounded on them. We also provide a characterisation of separability for quasi-Banach function spaces over the Euclidean space. Furthermore, we extend the classical Riesz--Fischer theorem to the context of quasinormed spaces and, as a consequence, obtain an alternative proof of completeness of quasi-Banach function spaces.
\end{abstract}

\date{\today}

\maketitle

\makeatletter
   \providecommand\@dotsep{2}
\makeatother

\section{Introduction}

While the theory of Banach function spaces has been extensively developed since their introduction by Luxemburg in \cite{Luxemburg55} (see \cite{BennettSharpley88} for an exhaustive treatment), their quasinormed counterparts remain surprisingly unexplored to the point where even the most fundamental results are missing. For example, even the most natural question of completeness has been answered only recently by Caetano, Gogatishvili and Opic in \cite{CaetanoGogatishvili16}. This lack of underlying theory constitutes a great obstacle in the field of function spaces, given that many natural and useful function spaces happen to be quasinormed while retaining many, if not all, of the remaining properties of Banach function norms. Examples of such spaces include Lebesgue spaces $L^p$ with $p<1$, Lorentz spaces $L^{p,q}$ with $q<1$, many cases of Lorentz-Karamata spaces (as introduced in \cite{EdmundsKerman00}), numerous cases of classical Lorentz spaces, and generalised Orlicz spaces $L_{\Phi}$ with a non-convex function $\Phi$. Further examples are the Lorentz spaces $L^{1, q}$, $q>1$, chief among them being the space $L^{1, \infty}$, also known as the weak $L^1$; the importance of this space stems from the fact that it captures the endpoint behaviour of many important operators of harmonic analysis, including the Hardy--Littlewood maximal operator, singular integral operators, and fractional integral operators.

It comes at no surprise that, thanks to their versatility and wide field of applications, quasi-Banach function spaces themselves attracted some attention in the recent years. For example, Newtonian spaces based on them were studied in \cite{MalyL12} and \cite{MalyL16}; Hardy spaces associated with them were the topic of \cite{SawanoHo17}, \cite{SunYang22}, \cite{WangYang20}, and \cite{WangYang23}; their interaction with a range of other concepts, including behaviour of some operators, was investigated in \cite{CaetanoGogatishvili16}, \cite{GuoZhao20}, \cite{HaleNaibo23}, \cite{Ho19}, \cite{Ho20}, \cite{Ho23}, \cite{Kolwicz18}, \cite{KolwiczLesnik19}, \cite{Nieraeth23}, and \cite{PanYang23}; and the paper \cite{LoristNieraeth23} covers some of their properties and alternative approaches to their definition. 

Another example are the spaces $X_{\alpha, \beta}$ introduced and studied in \cite{Ho16} and \cite{Ho20-Interpolation}, that are defined for any rearrangement invariant quasi-Banach function space $X$ over $((0, \infty), \lambda)$ and any pair of numbers $\alpha \in \mathbb{R}$ and $\beta \in [0, \infty)$ through the functional
\begin{equation*}
	\lVert f \rVert_{X_{\alpha, \beta}} = \lVert t^{-\frac{\alpha}{n}} f^*(t^{\beta}) \rVert_X,
\end{equation*}
where $f^*$ is the non-increasing rearrangement of $f$ (see Definition~\ref{DNIR} or \cite[Chapter~2, Definition~1.5]{BennettSharpley88}). Said spaces were introduced for the purpose of extending and generalising the Boyd's interpolation theorem to rearrangement invariant quasi-Banach function spaces and operators that are bounded from $L^p$ to $L^q$ with $p \neq q$.

Moreover, in connection with the investigation of traces of Sobolev functions as a particular case of upper Ahlfors measures, it was recently discovered in \cite{CianchiPick20} and \cite{CianchiPick20B} that it is very useful to work with scales of function spaces of the form $X^{\langle \alpha \rangle}$, studied later in \cite{Turcinova23}, where $X$ is a rearrangement-invariant Banach function space, $\alpha \in (0, \infty)$, and the space is defined through the functional 
\begin{equation*}
	\lVert u \rVert_{X^{\langle\alpha\rangle}} = \lVert (|u|^\alpha)^{**}(t)^{\frac{1}{\alpha}} \rVert_{\overline{X}(0,\mu(R))},
\end{equation*}
where $\overline{X}(0,\mu(R))$ is the representation space of $X$ and $f^{**}$ is the elementary maximal function of $f$ (see \cite[Chapter~2, Theorem~4.10]{BennettSharpley88} and \cite[Chapter~2, Definition~3.1]{BennettSharpley88}). It is not difficult to notice that this functional becomes for small $\alpha$ a quasinorm even if $X$ was a normed space. It is thus inevitable to work with quasinorms and the information about the basic properties of corresponding spaces comes very handy. 

This paper thus aims to fill in some of the gaps in the basic theory of quasi-Banach function spaces. We first provide a new version of the classical Riesz--Fischer condition and prove the corresponding version of Riesz--Fischer theorem. We then show that quasi-Banach function spaces satisfy the modified criterion. We thus obtain an alternative proof of the completeness of quasi-Banach function spaces, a result that was first obtained in \cite{CaetanoGogatishvili16}, but this result has a  much wider field of application, including a general version of Landau's resonance theorem and, most importantly, the result that an embedding between quasi-Banach function spaces is always continuous, i.e.~that $X \subseteq Y$ implies $X \hookrightarrow Y$. The latter property has proved to be extraordinarily useful in the study of a wide variety of problems in the theory of function spaces. We then turn our attention to separability and obtain a characterisation for the case when the underlying measure space is the $n$-dimensional Euclidean space. Finally, we show that the dilation operator is bounded on rearrangement invariant quasi-Banach function spaces.

The paper is structured as follows. We first provide the necessary theoretical background in Section~\ref{SP} and then state and prove our results in Section~\ref{CHQBFS}. More specifically, the extension of the Riesz--Fischer theorem is proved in  Section~\ref{SGRF}, the result that quasi-Banach function spaces have the generalised Riesz--Fischer property and its applications are contained in Section~\ref{SBPQBFS}, the characterisation of separability is obtained in Section~\ref{SS}, and the proof that the dilation operator is bounded on quasi-Banach function spaces is the content of Section~\ref{SBDO}.

\section{Preliminaries} \label{SP}
This section serves to establish the basic theoretical background necessary for stating and proving our results. The notation and definitions are intended to be as standard as possible. The usual reference for most of this theory are the books \cite{BennettSharpley88}, \cite{BenyaminiLindenstrauss00} and \cite{JohnsonLindenstraus03-25}.

Throughout this paper we will denote by $(R, \mu)$, and occasionally by $(S, \nu)$, some arbitrary (totally) sigma-finite measure space. Provided a $\mu$-measurable set $E \subseteq R$ we will denote its characteristic function by $\chi_E$. By $M(R, \mu)$ we will denote the set of all $\mu$-measurable functions defined on $R$ with values in $\mathbb{C} \cup \{ \infty\}$. As it is customary, we will identify functions that coincide $\mu$-almost everywhere. We will further denote by $M_0(R, \mu)$ and $M_+(R, \mu)$ the subsets of $M(R, \mu)$ containing, respectively, the functions finite $\mu$-almost everywhere and the functions with values in $[0, \infty]$. 

For brevity, we will abbreviate $\mu$-almost everywhere, $M(R, \mu)$, $M_0(R, \mu)$ and $M_+(R, \mu)$ to $\mu$-a.e., $M$, $M_0$ and $M_+$, respectively, when there is no risk of confusion.

In the special case when $\mu$ is the classical $n$-dimensional Lebesgue measure on $\mathbb{R}^n$ we will denote it by $\lambda^n$ and in the case when it is the $k$-dimensional Hausdorff measure on $\mathbb{R}^n$, where $k < n$, we will denote it by $\mathcal{H}^{k}$. We will also denote by $\mathbb{B}^n$ the $n$-dimensional unit ball and by $\mathbb{S}^{n-1}$ the $(n-1)$-dimensional unit sphere.

When $X, Y$ are two topological linear spaces, we will denote by $Y \hookrightarrow X$ that $Y \subseteq X$ and that the identity mapping $I : Y \rightarrow X$ is continuous.

Given any complex number $z \in \mathbb{C}$ we will denote its real part by $\textup{Re}(z)$ and its imaginary part by $\textup{Im}(z)$.

Note that in this paper we consider $0$ to be an element of $\mathbb{N}$.

\subsection{Norms and quasinorms} \label{SNQ}
Let us first present the standard definition of a norm.
\begin{definition} \label{DN}
	Let $X$ be a complex linear space. A functional $\lVert \cdot \rVert : X \rightarrow [0, \infty)$ will be called a norm if it satisfies the following conditions:
	\begin{enumerate}
		\item it is positively homogeneous, i.e.~$\forall a \in \mathbb{C} \; \forall x \in X : \lVert ax  \rVert = \lvert a \rvert \lVert x \rVert$,
		\item it satisfies $\lVert x \rVert = 0 \Leftrightarrow x = 0$  in $X$,
		\item it is subadditive, i.e.~$\forall x,y \in X : \lVert x+y \rVert \leq \lVert x \rVert + \lVert y \rVert$.
	\end{enumerate}
\end{definition}

In praxis this definition is sometimes too restrictive, in the sense that in some important cases the triangle inequality does not hold. This motivates the following natural generalisation.

\begin{definition} \label{DQ}
	Let $X$ be a complex linear space. A functional $\lVert \cdot \rVert : X \rightarrow [0, \infty)$ will be called a quasinorm if it satisfies the following conditions:
	\begin{enumerate}
		\item it is positively homogeneous, i.e.~$\forall a \in \mathbb{C} \; \forall x \in X : \lVert ax  \rVert = \lvert a \rvert \lVert x \rVert$,
		\item it satisfies $\lVert x \rVert = 0 \Leftrightarrow x = 0$  in $X$,
		\item \label{DQ3} there is a constant $C \geq 1$, called the modulus of concavity of $\lVert \cdot \rVert$, such that it is subadditive up to this constant, i.e.~$\forall x,y \in X \: : \: \lVert x+y \rVert \leq C(\lVert x \rVert + \lVert y \rVert)$.
	\end{enumerate}
\end{definition}

It is obvious that every norm is also a quasinorm with the modulus of concavity equal to $1$ and that every quasinorm with the modulus of concavity equal to $1$ is also a norm. 

It is a well-known fact that every norm defines a metrisable topology on $X$ and that it is continuous with respect to that topology. This is not true for quasinorms, but it can be remedied thanks to the Aoki--Rolewicz theorem which we list below. Further details can be found for example in \cite{JohnsonLindenstraus03-25} or in \cite[Appendix~H]{BenyaminiLindenstrauss00}.

\begin{theorem}
	Let $\lVert \cdot \rVert_1$ be a quasinorm over the linear space $X$. Then there is a  quasinorm $\lVert \cdot \rVert_2$ such that
	\begin{enumerate}
		\item there is a finite constant $C_0 > 0$ such that it holds for all $x \in X$ that
		\begin{equation*}
			C_0^{-1} \lVert x \rVert_1 \leq \lVert x \rVert_2 \leq C_0 \lVert x \rVert_1,
		\end{equation*}
		\item there is an $r \in (0, 1]$ such that it holds for all $x,y \in X$ that
		\begin{equation*}
			\lVert x+y \rVert_2^r \leq \lVert x \rVert_2^r + \lVert y \rVert_2^r.
		\end{equation*}
	\end{enumerate}
\end{theorem}

The direct consequence of this result is that every quasinorm defines a metrisable topology on $X$ and that the convergence in said topology is equivalent to the convergence with respect to the original quasinorm, in the sense that $x_n \rightarrow x$ in the induced topology if and only if $\lim_{n \rightarrow \infty} \lVert x_n - x \rVert = 0$. Furthermore, we may naturally extend the concept of completeness to quasinormed spaces.

\begin{definition} \label{DC}
	Quasinormed space $X$ equipped with a quasinorm $\lVert \cdot \rVert$ is said to be complete if and only if it holds that for any sequence $x_n$ of elements of $X$ that is Cauchy with respect to the quasinorm $\lVert \cdot \rVert$ there is some $x \in X$ such that $x_n \rightarrow x$.
\end{definition}

Natural question to ask is when do different quasinorms define equivalent topologies. It is an easy exercise to show that the answer is the same as in the case of norms, that is that two quasinorms are topologically equivalent if and only if they are equivalent in the following sense.

\begin{definition}
	Let $\lVert \cdot \rVert_1$ and  $\lVert \cdot \rVert_2$ be quasinorms over the linear space $X$. We say that $\lVert \cdot \rVert_1$ and  $\lVert \cdot \rVert_2$ are equivalent if there is some $C_0 > 0$ such that it holds for all $x \in X$ that
	\begin{equation*}
		C_0^{-1} \lVert x \rVert_1 \leq \lVert x \rVert_2 \leq C_0 \lVert x \rVert_1.
	\end{equation*}
\end{definition}

\subsection{Banach function norms and quasi-Banach function norms} \label{SBFNQN}

We now turn our attention to norms acting on spaces of functions. The approach taken here is the same as in \cite[Chapter 1, Section 1]{BennettSharpley88}, which means that it differs, at least formally, from that in Section~\ref{SNQ}.

The major definitions are of course those of Banach function norm and the corresponding Banach function space.

\begin{definition}
	Let $\lVert \cdot \rVert : M(R, \mu) \rightarrow [0, \infty]$ be a mapping satisfying $\lVert \, \lvert f \rvert \, \rVert = \lVert f \rVert$ for all $f \in M$. We say that $\lVert \cdot \rVert$ is a Banach function norm if its restriction to $M_+$ satisfies the following axioms:
	\begin{enumerate}[label=\textup{(P\arabic*)}, series=P]
		\item \label{P1} it is a norm, in the sense that it satisfies the following three conditions:
		\begin{enumerate}[ref=(\theenumii)]
			\item \label{P1a} it is positively homogeneous, i.e.\ $\forall a \in \mathbb{C} \; \forall f \in M_+ : \lVert a f \rVert = \lvert a \rvert \lVert f \rVert$,
			\item \label{P1b} it satisfies $\lVert f \rVert = 0 \Leftrightarrow f = 0$  $\mu$-a.e.,
			\item \label{P1c} it is subadditive, i.e.\ $\forall f,g \in M_+ \: : \: \lVert f+g \rVert \leq \lVert f \rVert + \lVert g \rVert$,
		\end{enumerate}
		\item \label{P2} it has the lattice property, i.e.\ if some $f, g \in M_+$ satisfy $f \leq g$ $\mu$-a.e., then also $\lVert f \rVert \leq \lVert g \rVert$,
		\item \label{P3} it has the Fatou property, i.e.\ if  some $f_n, f \in M_+$ satisfy $f_n \uparrow f$ $\mu$-a.e., then also $\lVert f_n \rVert \uparrow \lVert f \rVert $,
		\item \label{P4} $\lVert \chi_E \rVert < \infty$ for all $E \subseteq R$ satisfying $\mu(E) < \infty$,
		\item \label{P5} for every $E \subseteq R$ satisfying $\mu(E) < \infty$ there exists some finite constant $C_E$, dependent only on $E$, such that the inequality $ \int_E f \: d\mu \leq C_E \lVert f \rVert $ is true for all $f \in M_+$.
	\end{enumerate} 
\end{definition}

\begin{definition}
	Let $\lVert \cdot \rVert_X$ be a Banach function norm. We then define the corresponding Banach function space $X$ as the set
	\begin{equation*}
		X = \left \{ f \in M; \; \lVert f \rVert_X < \infty \right \}.
	\end{equation*}
	Furthermore, the normed linear space $(X, \lVert \cdot \rVert_X)$ will also be called a Banach function space.
\end{definition}

Just as with general norms, the triangle inequality is sometimes too strong a condition to require. This leads to the definition of quasi-Banach function norms and of the corresponding quasi-Banach function spaces which will be our main topic in this paper.

\begin{definition}
	Let $\lVert \cdot \rVert : M(R, \mu) \rightarrow [0, \infty]$ be a mapping satisfying $\lVert \, \lvert f \rvert \, \rVert = \lVert f \rVert$ for all $f \in M$. We say that $\lVert \cdot \rVert$ is a quasi-Banach function norm if its restriction to $M_+$ satisfies the axioms \ref{P2}, \ref{P3}, \ref{P4} of Banach function norms together with a weaker version of axiom \ref{P1}, namely
	\begin{enumerate}[label=\textup{(Q\arabic*)}]
		\item \label{Q1} it is a quasinorm, in the sense that it satisfies the following three conditions:
		\begin{enumerate}[ref=(\theenumii)]
			\item \label{Q1a} it is positively homogeneous, i.e.\ $\forall a \in \mathbb{C} \; \forall f \in M_+ : \lVert af \rVert = \lvert a \rvert \lVert f \rVert$,
			\item \label{Q1b} it satisfies  $\lVert f \rVert = 0 \Leftrightarrow f = 0$ $\mu$-a.e.,
			\item \label{Q1c} there is a constant $C \geq 1$, called the modulus of concavity of $\lVert \cdot \rVert$, such that it is subadditive up to this constant, i.e.
			\begin{equation*}
				\forall f,g \in M_+ : \lVert f+g \rVert \leq C(\lVert f \rVert + \lVert g \rVert).
			\end{equation*}
		\end{enumerate}
	\end{enumerate}
\end{definition}

\begin{definition}
	Let $\lVert \cdot \rVert_X$ be a quasi-Banach function norm. We then define the corresponding quasi-Banach function space $X$ as the set
	\begin{equation*}
		X = \left \{ f \in M; \; \lVert f \rVert_X < \infty \right \}.
	\end{equation*}
	Furthermore, the quasinormed linear space $(X, \lVert \cdot \rVert_X)$ will also be called a quasi-Banach function space.
\end{definition}

The most natural examples of quasi-Banach function spaces are the classical Lebesgue spaces $L^p$ with $p \in (0, 1)$. To provide a natural example of quasi-Banach function space that also satisfies the axiom \ref{P5} we turn to the Lorentz spaces $L^{p,q}$ with $p>1$ and $q<1$. The quasi-Banach function norm they are induced from is, in our case $q<1$, of the form
\begin{equation*}
	\lVert f \rVert_{p,q} = \left ( \int_{0}^{\infty} t^{\frac{q}{p}-1} \left ( f^*(t) \right )^q \: dt \right )^{\frac{1}{q}},
\end{equation*}
where $f^*$ denotes the non-increasing rearrangement of $f$ as defined in Definition~\ref{DNIR}.

Let us now list here two of the most important properties of Banach function spaces which we will generalise in Section~\ref{CHQBFS}. Proofs of the standard versions presented below can be found in \cite[Chapter 1, Section 1]{BennettSharpley88}. 

\begin{theorem} \label{TBFSC}
	Let $(X, \lVert \cdot \rVert_X)$ be a Banach function space. Then it is a Banach space.
\end{theorem}

\begin{theorem} \label{TBFSE}
	Let $(X, \lVert \cdot \rVert_X)$ and $(Y, \lVert \cdot \rVert_Y)$ be Banach function spaces. If $X \subseteq Y$ then also $X \hookrightarrow Y$.
\end{theorem}

An important concept in the theory of quasi-Banach function spaces is the absolute continuity of the quasinorm.

\begin{definition}
	Let $(X, \lVert \cdot \rVert_X)$ be a quasi-Banach function space. We say that a function $f \in X$ has absolutely continuous quasinorm if it holds that $\lVert f \chi_{E_k} \rVert_X \rightarrow 0$ whenever $E_k$ is a sequence of $\mu$-measurable subsets of $R$ such that $\chi_{E_k} \rightarrow 0$ $\mu$-a.e.
	
	If every $f \in X$ has absolutely continuous quasinorm we further say that the space $X$ itself has absolutely continuous quasinorm.
\end{definition}

This concept is important, because it is deeply connected to separability and also reflexivity of Banach function spaces, for details see \cite[Chapter~1, Sections~3, 4 and 5]{BennettSharpley88}. As we will show in Section~\ref{SS}, at least in the case of separability this connection extends to the context of quasi-Banach function spaces. For now, we present here the following two elementary observations that we will use later on. Their proofs are exactly the same as the proofs of analogous statements in \cite[Chapter~1, Section~3]{BennettSharpley88}.

\begin{proposition}
	Let $(X, \lVert \cdot \rVert_X)$ be a quasi-Banach function space and let $f \in X$. Then $f$ has absolutely continuous quasinorm if and only if $\lVert f \chi_{E_k} \rVert \downarrow 0$ for every sequence $E_k$ of subsets of $R$ such that $\chi_{E_k} \downarrow 0$ $\mu$-a.e.
\end{proposition}

\begin{proposition}
	Let $(X, \lVert \cdot \rVert_X)$ be a quasi-Banach function space and let $f \in X$. Suppose that $f$ has absolutely continuous quasinorm. Then for every $\varepsilon > 0$ there is a $\delta > 0$ such that given any $\mu$-measurable set $E \subseteq R$ satisfying $\mu(E) < \delta$ it holds that $\lVert f \chi_E \rVert_X < \varepsilon$.
\end{proposition}

Another important concept is that of an associate space. The detailed study of associate spaces of Banach function spaces can be found in \cite[Chapter 1, Sections 2, 3 and 4]{BennettSharpley88}. We will approach the issue in a slightly more general way. The definition of an associate space requires no assumptions on the functional defining the original space.

\begin{definition} \label{DAS}
	Let $\lVert \cdot \rVert_X: M \to [0, \infty]$ be some non-negative functional and put
	\begin{equation*}
		X = \{ f \in M; \; \lVert f \rVert_X < \infty \}.
	\end{equation*} 
	Then the functional $\lVert \cdot \rVert_{X'}$ defined for $f \in M$ by 
	\begin{equation*}
	\lVert f \rVert_{X'} = \sup_{g \in X} \frac{1}{\lVert g \rVert_X} \int_R \lvert f g \rvert \: d\mu, \label{DAS1}
	\end{equation*}
	where we interpret $\frac{0}{0} = 0$ and $\frac{a}{0} = \infty$ for any $a>0$, will be called the associate functional of $\lVert \cdot \rVert_X$ while the set
	\begin{equation*}
		X' = \left \{ f \in M; \; \lVert f \rVert_{X'} < \infty \right \}
	\end{equation*}
	will be called the associate space of $X$.
\end{definition}

As suggested by the notation, we will be interested mainly in the case when  $\lVert \cdot \rVert_X$ is at least a quasinorm, but we wanted to indicate that such assumption is not necessary for the definition. In fact, it is not even required for the following result, which is the Hölder inequality for associate spaces.

\begin{theorem} \label{THAS}
	Let $\lVert \cdot \rVert_X: M \to [0, \infty]$ be some non-negative functional and denote by $\lVert \cdot \rVert_{X'}$ its associate functional. Then it holds for all $f,g \in M$ that
	\begin{equation*}
		\int_R \lvert f g \rvert \: d\mu \leq \lVert g \rVert_X \lVert f \rVert_{X'}
	\end{equation*}
	provided that we interpret $0 \cdot \infty = -\infty \cdot \infty = \infty$ on the right-hand side.
\end{theorem}

The convention concerning the products at the end of this theorem is necessary precisely because we put no restrictions on $\lVert \cdot \rVert_X$ and thus there occur some pathological cases which need to be taken care of. Specifically, $0 \cdot \infty = \infty$ is required because we allow $\lVert g \rVert_X = 0$ even for non-zero $g$ while $-\infty \cdot \infty = \infty$ is required because Definition~\ref{DAS} allows $X = \emptyset$ which implies $\lVert f \rVert_{X'} = \sup \emptyset = -\infty$.

In order for the associate functional to be well behaved some assumptions on $\lVert \cdot \rVert_X$ are needed. The following result, due to Gogatishvili and Soudsk{\'y} in \cite{GogatishviliSoudsky14}, provides a sufficient condition for the associate functional to be a Banach function norm.

\begin{theorem} \label{TFA}
	Let $\lVert \cdot \rVert_X : M \to [0, \infty]$ be a functional that satisfies the axioms \ref{P4} and \ref{P5} from the definition of Banach function spaces and which also satisfies for all $f \in M$ that $\lVert f \rVert_X$ = $\lVert \, \lvert f \rvert \, \rVert_X$. Then the functional $\lVert \cdot \rVert_{X'}$ is a Banach function norm. In addition, $\lVert \cdot \rVert_X$ is equivalent to a Banach function norm if and only if there is some constant $C >0$ such that it holds for all $f \in M$ that
	\begin{equation} 
	C^{-1} \lVert f \rVert_{X''} \leq \lVert f \rVert_X \leq C \lVert f \rVert_{X''}, \label{TFA1}
	\end{equation}	
	where $\lVert \cdot \rVert_{X''}$ denotes the associate functional of $\lVert \cdot \rVert_{X'}$.
\end{theorem}

As a special case, we get that the associate functional of any quasi-Banach function space that also satisfies the axiom \ref{P5} is a Banach function norm. This has been observed earlier in \cite[Remark~2.3.(iii)]{EdmundsKerman00}.

Additionally, if $\lVert \cdot \rVert_X$ is a Banach function norm then \eqref{TFA1} holds with constant one. This is a classical result of Lorentz and Luxemburg, proof of which can be found for example in \cite[Chapter~1, Theorem~2.7]{BennettSharpley88}.

\begin{theorem} \label{TDAS}
	Let $\lVert \cdot \rVert_X$ be a Banach function norm, then $\lVert \cdot \rVert_X = \lVert \cdot \rVert_{X''}$ where $\lVert \cdot \rVert_{X''}$ is the associate functional of $\lVert \cdot \rVert_{X'}$.
\end{theorem}

Let us point out that even in the case when $\lVert \cdot \rVert_X$, satisfying the assumptions of Theorem~\ref{TFA}, is not equivalent to any Banach function norm we still have one embedding, as formalised in the following statement. The proof is an easy exercise.

\begin{proposition} \label{PESSAS}
	Let $\lVert \cdot \rVert_X$ satisfy the assumptions of Theorem~\ref{TFA}. Then it holds for all $f \in M$ that
	\begin{equation*}
		\lVert f \rVert_{X''} \leq  \lVert f \rVert_X,
	\end{equation*}
	where $\lVert \cdot \rVert_{X''}$ denotes the associate functional of $\lVert \cdot \rVert_{X'}$.
\end{proposition}

\subsection{Non-increasing rearrangement}
In this section we present the concept of the non-increasing rearrangement of a function and state some of its properties that will be important in the last part of the paper. We proceed in accordance with \cite[Chapter 2]{BennettSharpley88}.

The first step is to introduce the distribution function which is defined as follows.

\begin{definition}	
	The distribution function $\mu_f$ of a function $f \in M$ is defined for $s \in [0, \infty)$ by
	\begin{equation*}
		\mu_f(s) = \mu(\{ t \in R; \; \lvert f(t) \rvert > s \}).
	\end{equation*}	
\end{definition}

The non-increasing rearrangement is then defined as the generalised inverse of the distribution function.

\begin{definition} \label{DNIR}
	The non-increasing rearrangement $f^*$ of function $f \in M$ is defined for $t \in [0, \infty)$ by
	\begin{equation*}
		f^*(t) = \inf \{ s \in [0, \infty); \; \mu_f(s) \leq t \}.
	\end{equation*}
\end{definition}

For the basic properties of the distribution function and the non-increasing rearrangement, with proofs, see \cite[Chapter 2, Proposition 1.3]{BennettSharpley88} and \cite[Chapter 2, Propositin 1.7]{BennettSharpley88} respectively. We consider those basic properties to be classical and well known and we will be using them without further explicit reference.

An important class of Banach function spaces are rearrangement invariant Banach function spaces. Below we provide a slightly more general definition which also allows for quasi-Banach function spaces.

\begin{definition}
	Let $\lVert \cdot \rVert_X$ be a quasi-Banach function norm. We say that $\lVert \cdot \rVert_X$ is rearrangement invariant, abbreviated r.i., if $\lVert f\rVert_X = \lVert g \rVert_X$ whenever $f, g \in M$ satisfy $f^* = g^*$.
		
	Furthermore, if the above condition holds, the corresponding space $(X,\lVert \cdot \rVert_X)$ will be called rearrangement invariant too.
\end{definition}

An important property of r.i.~Banach function spaces over $(\mathbb{R}^n, \lambda^n)$ is that the dilation operator is bounded on those spaces, as stated in the following theorem. Standard proof uses interpolation, details can be found for example in \cite[Chapter 3, Proposition 5.11]{BennettSharpley88} or \cite[Section~2.b.]{LindenstraussTzafriri79}. Alternatively, a direct proof can be obtained by the means of the Lorentz--Luxemburg theorem (Theorem~\ref{TDAS}).

\begin{theorem} \label{TDRIS}
	Let $n \in \mathbb{N}$ and let $(X,\lVert \cdot \rVert_X)$ be an r.i.~Banach function space over $(\mathbb{R}^n, \lambda^n)$. Consider the dilation operator $D_a$ defined on $M(\mathbb{R}^n, \lambda^n)$ by
	\begin{equation*}
		D_af(s) = f(as).
	\end{equation*}
	Then $D_a: X \rightarrow X$ is a bounded operator.
\end{theorem}

We extend this result to the context of r.i.~quasi-Banach spaces in Theorem~\ref{TBD}. To this end, we employ the concept of radial non-increasing rearrangement which we present below.

\begin{definition} \label{DRNIR}
	Let $n \in \mathbb{N}, n \geq 1$. The radial non-increasing rearrangement $f^{\star}$ of function $f \in M(\mathbb{R}^n, \lambda^n)$ is defined for $x \in \mathbb{R}^n$ by
	\begin{equation*}
		f^{\star}(x) = f^*(\alpha_n \lvert x \rvert^n),
	\end{equation*}
	where the constant $\alpha_n$ is defined by
	\begin{equation} \label{DRNIR1}
		\alpha_n = \frac{1}{n} \mathcal{H}^{n-1}\left( \mathbb{S}^{n-1} \right) = \lambda^n \left (\mathbb{B}^n \right).
	\end{equation}
\end{definition}

It is obvious that $f^{\star}$ is radially symmetrical and non-increasing in $\lvert x \rvert$. Moreover, it is an easy exercise to check that $(f^{\star})^* = f^*$ and also that the operation $f \mapsto f^{\star}$ commutes with the dilation operator $D_a$, i.e.
\begin{equation*}
	(D_a f )^{\star} = D_a(f^{\star}).
\end{equation*}
Those two properties are in fact the motivation behind \eqref{DRNIR1}.

\section{Quasi-Banach function spaces} \label{CHQBFS}

The core observation of this section is the following lemma. Although it is in fact quite simple to prove, it is extremely useful as it provides the critical insight needed in order to generalise the standard proofs from the theory of normed spaces.

\begin{lemma} \label{NT}
	Let $X$ be a quasinormed space equipped with the quasinorm $\lVert \cdot \rVert_X$ and denote by $C$ its modulus of concavity. Let $x_n$ be a sequence of points in $X$. Then
	\begin{equation*}
		\left \lVert \sum_{n = 0}^{N} x_n \right \rVert_X \leq \sum_{n = 0}^{N} C^{n+1} \lVert x_n \rVert_X
	\end{equation*}
	for every $N \in \mathbb{N}$.
\end{lemma}

\begin{proof}
	The estimate follows by fixing an $N \in \mathbb{N}$ and using the triangle inequality, up to a multiplicative constant, of $\lVert \cdot \rVert_X$ $(N+1)$ times to obtain
	\begin{equation*}
		\left \lVert \sum_{n = 0}^{N} x_n \right \rVert_X \leq C \lVert x_0 \rVert_X + C \left \lVert \sum_{n = 1}^{N} x_n \right \rVert_X \leq \dots  \leq \sum_{n = 0}^{N} C^{n+1} \lVert x_n \rVert_X.
	\end{equation*}
\end{proof}

\subsection{Generalised Riesz--Fischer theorem} \label{SGRF}

Firstly, we use Lemma~\ref{NT} to prove a generalised version of the classical Riesz--Fischer theorem. 

Let us first define a generalisation of the classical Riesz--Fischer property.

\begin{definition}
	Let $X$ be a quasinormed space equipped with the quasinorm $\lVert \cdot \rVert_X$ and let $C \in [1,\infty)$. We say that $X$ has the generalised Riesz--Fischer property with constant $C$ if for every sequence $x_n$ of points in $X$ that satisfies
	\begin{equation*}
		\sum_{n = 0}^{\infty} C^{n+1} \lVert x_n \rVert_X < \infty
	\end{equation*}
	there is a point $x \in X$ such that
	\begin{equation*}
		\lim_{N \rightarrow \infty} \sum_{n = 0}^{N} x_n = x
	\end{equation*}
	in the quasinormed topology of $X$.
	
	We further say that $X$ has the strong generalised Riesz--Fischer property, if every sequence $x_n$ of points in $X$ for which there exists a constant $\widetilde{C} \in (1,\infty)$ such that
	\begin{equation*}
		\sum_{n = 0}^{\infty} \widetilde{C}^{n+1} \lVert x_n \rVert_X < \infty
	\end{equation*}
	satisfies, that	there is a point $x \in X$ such that
	\begin{equation*}
		\lim_{N \rightarrow \infty} \sum_{n = 0}^{N} x_n = x
	\end{equation*}
	in the quasinormed topology of $X$.
\end{definition}

It is obvious that the generalised Riesz--Fischer property with constant $C=1$ is simply the classical Riesz--Fischer property, that the property gets weaker as $C$ becomes larger, and that the strong generalised Riesz--Fischer property implies the generalised Riesz--Fischer property for every constant $C > 1$. However, as follows from the following theorem, the only special case is the classical Riesz--Fischer property, everything else is mutually equivalent and also equivalent to the completeness of the quasinormed space in question.

\begin{theorem} \label{QRF}
	Let $X$ be a quasinormed space equipped with the quasinorm $\lVert \cdot \rVert_X$ and denote by $C_X$ its modulus of concavity. Then the following statements are equivalent:
	\begin{enumerate}
		\item $X$ is complete. \label{QRFi}
		\item $X$ has the generalised Riesz--Fischer property with some constant $C \in [1, \infty)$. \label{QRFii}
		\item $X$ has the generalised Riesz--Fischer property with constant $C_X$. \label{QRFiii}
		\item $X$ has the generalised Riesz--Fischer property with every constant $C \in (1, \infty)$. \label{QRFiv}
		\item $X$ has the strong generalised Riesz--Fischer property. \label{QRFv}
	\end{enumerate} 
\end{theorem}

\begin{proof}
	It is evident that \ref{QRFv} $\implies$ \ref{QRFiv} $\implies$ \ref{QRFiii} $\implies$ \ref{QRFii}.
	
	\ref{QRFii} $\implies$ \ref{QRFi}: Suppose that the space $X$ has the Riesz--Fischer property with arbitrary constant $C \in [1, \infty)$ and fix some Cauchy sequence $x_n$. Proceed to find a non-decreasing and unbounded sequence of natural numbers $k_n$ such that it holds for all natural numbers $i, j \geq k_n$ that
	\begin{equation*}
		\lVert x_i - x_j \rVert_X \leq (2C)^{-n-2}.
	\end{equation*}
	Now, let us consider sequence $y_n$ of points in $X$ defined by
	\begin{align*}
		y_0 &= x_{k_0}, \\
		y_n &= x_{k_n} - x_{k_{n-1}} \text{ for } n \geq 1.
	\end{align*}
	Then the sequence $y_n$ satisfies
	\begin{equation*} 
		\sum_{n = 0}^{\infty} C^{n+1} \lVert y_n \rVert_X \leq C \lVert x_{k_0} \rVert + \sum_{n = 1}^{\infty} C^{n+1} \lVert x_{k_n} - x_{k_{n-1}} \rVert_X \leq C \lVert x_{k_0} \rVert + \sum_{n = 1}^{\infty} 2^{-n-1} < \infty,
	\end{equation*}
	which means that, by our assumption on $X$, there is some limit $y \in X$ of the sequence $\sum_{n = 0}^{N} y_n$. Since
	\begin{equation*}
		\sum_{n = 0}^{N} y_n = x_{k_0} + \sum_{n = 1}^{N} x_{k_n} - x_{k_{n-1}} = x_{k_N}
	\end{equation*}
	we have shown that the sequence $x_n$ has a convergent subsequence with limit $y$. Because $x_n$ is Cauchy, the standard argument yields that $y$ is also the limit of $x_n$. 
	
	\ref{QRFi} $\implies$ \ref{QRFv}: Suppose that $X$ is complete and that $x_n$ is a sequence of points in $X$ for which there exists some constant $\widetilde{C} \in (1,\infty)$ such that
	\begin{equation*}
		\sum_{n = 0}^{\infty} \widetilde{C}^{n+1} \lVert x_n \rVert_X < \infty.
	\end{equation*}
	Then we might find some $K \in \mathbb{N}$, $K>0$, such that $\widetilde{C}^K \geq C_X$. Consider now for $i \in \{0, \dots, K-1\}$ the sequences $y_{i,n}$ defined pointwise (for $n \in \mathbb{N}$) as
	\begin{equation*}
		y_{i,n} = x_{Kn + i}.
	\end{equation*}
	That is, we have decomposed $x_n$ uniformly into $K$ subsequences. Each $y_{i,n}$ now satisfies
	\begin{equation} \label{QRF_1}
		\sum_{n=0}^{\infty} C_X^{n+1} \lVert y_{i,n} \rVert_X \leq \widetilde{C}^K \sum_{n=0}^{\infty} \widetilde{C}^{Kn} \lVert x_{Kn + i} \rVert_X \leq \widetilde{C}^K \sum_{n = 0}^{\infty} \widetilde{C}^{n+1} \lVert x_n \rVert_X < \infty.
	\end{equation}
	Next, we want to prove that the partial sums of the original series are Cauchy. Consider thus some arbitrary natural numbers $N \leq M$ and put as $N_0$ the largest natural number such that $KN_0 \leq N$, while $M_0$ will be the smallest natural number such that $M \leq KM_0$. Then, by decomposing the elements $x_n$, $N \leq n \leq M$, into $K$ groups according to their membership in the subsequences $y_{i,n}$ and applying first $K$ times the triangle inequality up to a constant and then Lemma~\ref{NT}, we obtain
	\begin{equation*}
		\left \lVert \sum_{n = 0}^N x_n - \sum_{n = 0}^M x_n \right \rVert_X \leq C_X^K \sum_{i=0}^{K-1} \sum_{n = N_0}^{M_0} C_X^{n - N_0 + 1} \lVert y_{i,n} \rVert_X \leq C_X^K \sum_{i=0}^{K-1} \sum_{n = N_0}^{\infty} C_X^{n + 1} \lVert y_{i,n} \rVert_X.
	\end{equation*}
	Thanks to \eqref{QRF_1}, we can make the right-hand side arbitrarily small by taking $N_0$ sufficiently large. Hence, the partial sums are Cauchy and our assumption of completeness ensures that the series converges.
\end{proof}

This result has a significant overlap with \cite[Theorem~1.1]{Maligranda04} (we would like to thank Professor Maligranda for making us aware of this). Significantly restricted versions were obtained earlier, see for example \cite[Lemma~101.1]{Zaanen83} or \cite{HalperinLuxemburg56}.

\subsection{Basic properties of quasi-Banach function norms} \label{SBPQBFS}

We now turn our attention to quasi-Banach function spaces as defined in Section~\ref{SBFNQN} and show that they have the same basic properties as their normed counterparts. For the proofs of the classical versions of these results see \cite[Chapter~1, Section~1]{BennettSharpley88}.

The first result relates quasi-Banach function spaces with the set of simple functions and $M_0$. 

\begin{theorem} \label{TEiMF}
	Let $(X, \lVert \cdot \rVert_X)$ be a quasi-Banach function space. Then $X$ is a linear space satisfying 
	\begin{equation*}
		S \subseteq X \hookrightarrow M_0
	\end{equation*}
	where $S$ denotes the set of all simple functions (supported on a set of finite measure) and $M_0$ is equipped with the topology of convergence in measure on the sets of finite measure. 
\end{theorem}
 
Our proof of this result is inspired by that of \cite[Chapter~II, Section~2, Theorem~1]{KreinPetunin82}. We would also like to note that \cite[Chapter~1, Theorem~1.4]{BennettSharpley88} contains an easier proof for the special case when $\lVert \cdot \rVert_X$ satisfies \ref{P5} and that the set theoretical inclusion $X \subseteq \mathcal{M}_0$ is contained in \cite[Lemma~2.4]{MizutaNekvinda15} (in both cases the results are formulated with different assumptions but the proofs translate verbatim to our setting).
	
We will need the following measure-theoretical lemma.

\begin{lemma} \label{LST}
	Let $E \subseteq \mathcal{R}$ satisfy $0 < \mu(E) < \infty$, let $\delta \in (0, \mu(E))$ be fixed, and let $F_n$ be a sequence of subsets of $E$ satisfying $\mu(F_n) > \delta$ for every $n \in \mathbb{N}$.
	Then the pointwise sum $\sum_{n=0}^{\infty} \chi_{F_n}$ satisfies
	\begin{equation*}
		\mu \left ( \left \{ \sum_{n=0}^{\infty} \chi_{F_n} = \infty \right \} \right ) > 0.
	\end{equation*}
\end{lemma}

\begin{proof}
	Assume contrary, i.e.~that we have a sequence $F_n$ of subsets of $E$ that satisfies  $\mu(F_n) > \delta$ for every $n \in \mathbb{N}$, where $\delta$ is the number fixed in the formulation, while
	\begin{equation*}
		\mu \left ( \left \{ \sum_{n=0}^{\infty} \chi_{F_n} = \infty \right \} \right ) = 0.
	\end{equation*}
	Then for every $\varepsilon \in (0, \mu(E))$ there is some $N_{\varepsilon} \in \mathbb{N}$ such that
	\begin{align*}
		\mu \left ( \left \{ \sum_{n=0}^{\infty} \chi_{F_n} > N_{\varepsilon} \right \} \right ) & < \frac{\varepsilon}{2}.
	\end{align*}
	Put $\varepsilon = \delta$. Then, by the monotone convergence theorem,
	\begin{equation*}
		\begin{split}
			N_{\delta} \mu(E)  &\geq \int_{\left \{ \sum_{n=0}^{\infty} \chi_{F_n} \leq N_{\delta} \right \} } \sum_{n=0}^{\infty} \chi_{F_n} \: d\mu \\
			&= \sum_{n=0}^{\infty} \int_{\left \{ \sum_{n=0}^{\infty} \chi_{F_n} \leq N_{\delta} \right \} }  \chi_{F_n} \: d\mu\\
			&= \sum_{n=0}^{\infty} \mu\left ( F_n \cap \left \{ \sum_{n=0}^{\infty} \chi_{F_n} \leq N_{\delta} \right \} \right ) \\
			&\geq \sum_{n=0}^{\infty} \left ( \mu(F_n) - \frac{\delta}{2} \right ).
		\end{split}
	\end{equation*}
	It follows that there is some $n_0 \in \mathbb{N}$ such that $\mu(F_{n_0}) < \delta$. 
\end{proof}

\begin{proof}[Proof of Theorem~\ref{TEiMF}]
	The inclusion $S \subseteq X$ is trivial so we will only prove the embedding $X \hookrightarrow M_0$. We first note that $X \subseteq \mathcal{M}_0$ in the set-theoretical sense, as follows from the following argument.
	
	Let $f \in \mathcal{M}$ and let $E = \{\lvert f \rvert = \infty\}$. If $\mu(E) > 0$ then it follows from the part \ref{Q1b} of \ref{Q1} that $\lVert \chi_E \rVert_X > 0$ (regardless of whether $\mu(E)$ is finite; we neither claim nor need finiteness of the quasinorm). As \ref{P2} and part \ref{Q1a} of \ref{Q1} imply $\lVert f \rVert_X \geq n \lVert \chi_E \rVert_X$ for every $n \in \mathbb{N}$, we conclude that $\lVert f \rVert_X = \infty$.
	
	It remains to show that the embedding is continuous, i.e.~that for every sequence $f_n$ in $X$ we have that $\lVert f_n \rVert_X \to 0$ as $n \to 0$ implies that $f_n$ converge to zero in measure on every set of finite measure. We prove this statement by contradiction.
	
	Let $C$ be the modulus of concavity of $\lVert \cdot \rVert_X$. Assume that $f_n$ is a sequence in $X$ such that $\lVert f_n \rVert_X \to 0$ as $n \to 0$ and that there is a set $E_0 \subseteq R$ with $0 < \mu(E_0) < \infty$, $\varepsilon \in (0, \infty)$, and $\delta \in (0, \infty)$ such that the sets 
	\begin{align*}
		F_n = \{x \in E_0 ; \; \lvert f_n(x) \rvert > \varepsilon \}
	\end{align*}
	satisfy $\mu(F_n) > \delta$ for every $n \in \mathbb{N}$. We may assume without loss of generality that $\lVert f_n \rVert_X \leq C^{-n-1} 2^{-n-1}$ for every $n \in \mathbb{N}$. Then it follows from Lemma~\ref{NT}, the properties \ref{P2} and \ref{P3} of $\lVert \cdot \rVert_X$, and part \ref{Q1a} of its property \ref{Q1} that the pointwise sum $\sum_{n=0}^{\infty} \chi_{F_n}$ satisfies
	\begin{equation*}
		\left \lVert \sum_{n=0}^{\infty} \chi_{F_n} \right \rVert_X = \lim_{N \to \infty} \left \lVert \sum_{n=0}^{N} \chi_{F_n} \right \rVert_X \leq \lim_{N \to \infty} \sum_{n=0}^{N} C^{n+1} \left  \lVert  \chi_{F_n} \right \rVert_X \leq \lim_{N \to \infty} \sum_{n=0}^{N} C^{n+1} \left  \lVert  \frac{1}{\varepsilon} f_n \right \rVert_X \leq \frac{1}{\varepsilon}.
	\end{equation*}
	Hence, $\sum_{n=0}^{\infty} \chi_{F_n} \in X$. Note that the sum is defined pointwise, we do not make any claims about convergence in $X$. The already proved set-theoretical inclusion $X \subseteq \mathcal{M}_0$ now asserts that $\sum_{n=0}^{\infty} \chi_{F_n}$ is finite almost everywhere; however, this statement contradicts Lemma~\ref{LST}.
\end{proof}

The next result is a version of Fatou's lemma.

\begin{lemma} \label{LF}
	Let $(X, \lVert \cdot \rVert_X)$ be a quasi-Banach function space. Consider a sequence $f_n$ of functions in $X$ and $f \in X$. Then the following two assertions hold.
	\begin{enumerate}
		\item If $0 \leq f_n \uparrow f$ $\mu$-a.e., then either $f \notin X$ and $\lVert f_n \rVert_X \uparrow \infty$ or $f \in X$ and $\lVert f_n \rVert_X \uparrow \lVert f \rVert_X$. \label{LFp1}
		\item If $f_n \rightarrow f$ $\mu$-a.e.~and $\liminf_{n \rightarrow \infty} \lVert f_n \rVert_X < \infty$, then $f \in X$ and
		\begin{equation*}
			\lVert f \rVert_X \leq \liminf_{n \rightarrow \infty} \lVert f_n \rVert_X.
		\end{equation*} \label{LFp2}
	\end{enumerate}
\end{lemma}

The proof is omitted since it is the same as in the classical case which can be found in \cite[Chapter 1, Lemma 1.5]{BennettSharpley88}.

The following result establishes the Riesz--Fischer property, with an appropriate constant, of quasi-Banach function spaces. We believe this result to be interesting in its own right because its applications are not limited to proving completeness.

\begin{theorem} \label{TC}
	Let $(X, \lVert \cdot \rVert_X)$ be a quasi-Banach function space. Denote by $C$ the modulus of concavity of $\lVert \cdot \rVert_X$. Then $X$ has the Riesz--Fischer property with constant $C$.
	
	Furthermore, let $f_n$ be a sequence in $X$ such that
	\begin{equation*}
		\sum_{n = 0}^{\infty} C^{n+1} \lVert f_n \rVert_X < \infty
	\end{equation*}
	and put $f = \sum_{n = 0}^{\infty} f_n$ in $X$. Then also $f = \sum_{n = 0}^{\infty} f_n$  $\mu$-a.e.
\end{theorem}

\begin{proof}
	Fix some sequence $f_n$ in $X$ such that
	\begin{equation}
		\sum_{n = 0}^{\infty} C^{n+1} \lVert f_n \rVert_X < \infty. \label{TC1}
	\end{equation}
	Denote by $t$ and $t_N$ the following pointwise sums:
	\begin{align*}
		t &= \sum_{n=0}^{\infty} \lvert f_n \rvert, \\
		t_N &= \sum_{n=0}^{N} \lvert f_n \rvert.
	\end{align*}
	Then $t_N \uparrow t$ and since it holds by Lemma~\ref{NT} that
	\begin{equation*}
		\lVert t_N \rVert_X \leq \sum_{n = 0}^{N} C^{n+1} \lVert f_n \rVert_X \leq \sum_{n = 0}^{\infty} C^{n+1} \lVert f_n \rVert_X < \infty,
	\end{equation*}
	we get by part~\ref{LFp1} of Lemma~\ref{LF} that $t \in X$. Thanks to Theorem~\ref{TEiMF} the series $\sum_{n=0}^{\infty} \lvert f_n \rvert$ converges almost everywhere and therefore the series $\sum_{n=0}^{\infty} f_n$ does too.
	
	Denote now by $f$ and $s_N$ the following pointwise sums:
	\begin{align*}
		f &= \sum_{n=0}^{\infty} f_n, \\
		s_N &= \sum_{n=0}^{N} f_n.
	\end{align*}
	Then $s_N \rightarrow f$ $\mu$-a.e., hence, for any $M$, we get that $s_N - s_M \rightarrow f - s_M$ $\mu$-a.e.~as $N \rightarrow \infty$. Furthermore, using Lemma~\ref{NT} again, we get that
	\begin{equation*}
		\liminf_{N \rightarrow \infty} \lVert s_N - s_M \rVert_X \leq \liminf_{N \rightarrow \infty} \sum_{n = M+1}^{N} C^{n+1} \lVert f_n \rVert_X \leq \sum_{n = M+1}^{\infty} C^{n+1} \lVert f_n \rVert_X,
	\end{equation*}
	which tends to $0$ as $M \rightarrow \infty$ thanks to \eqref{TC1}. Therefore, if follows from part~\eqref{LFp2} of Lemma~\ref{LF} that $f - s_M \in X$ (which implies that $f \in X$ too) and also that $\lVert f - s_M \rVert_X \rightarrow 0$ as $M \rightarrow \infty$.
\end{proof}

The most obvious application of this result is the completeness of quasi-Banach function spaces, which follows via Theorem~\ref{QRF}. This result was first obtained by Caetano, Gogatishvili and Opic in \cite{CaetanoGogatishvili16} using a different method.

\begin{corollary}
	Let $(X,\lVert \cdot \rVert_X)$ be a quasi-Banach function space. Then it is complete.
\end{corollary}

A second application of Theorem~\ref{TC} is the following result that characterises the boundedness of those quasilinear operators that are in some sense compatible with the structure of quasi-Banach function spaces.

\begin{theorem} \label{TOQBFS}
	Let $(X,\lVert \cdot \rVert_X)$ and $(Y,\lVert \cdot \rVert_Y)$ be quasi-Banach function spaces over $M(R, \mu)$ and $M(S, \nu)$, respectively, and let $T: M(R, \mu) \rightarrow M(S, \nu)$ be a quasilinear operator such that the following two conditions hold:
	\begin{enumerate}
		\item \label{T1}
		There is a constant $c_1 > 0$ such that it holds for all $f \in M(R, \mu)$ that $\lvert T(f) \rvert \leq c_1 T(\lvert f \rvert)$ $\nu$-a.e.~on $S$.
		\item \label{T2}
		There is a constant $c_2 > 0$ such that it holds for all $f, g \in M(R, \mu)$ satisfying $\lvert f \rvert \leq \lvert g \rvert$ $\mu$-a.e.~on $R$ that $\lvert T(f) \rvert \leq c_2 \lvert T(g) \rvert$ $\nu$-a.e.~on $S$.
	\end{enumerate}
	Then $T:X \rightarrow Y$ is bounded, in the sense that there is a finite constant $C_T >0$ such that $\lVert T(f) \rVert_Y \leq C_T \lVert f \rVert_X$ for all $f \in X$, if and only if $T(f) \in Y$ for all $f \in X$.
\end{theorem}

\begin{proof}
	We will only prove the sufficiency since the necessity is obvious.
	
	Denote by $C$ the modulus of concavity of both $\lVert \cdot \rVert_X$ and $\lVert \cdot \rVert_Y$ and suppose that no finite constant $C_T$ satisfies $\lVert T(f) \rVert_Y \leq C_T \lVert f \rVert_X$ for all $f \in X$. Then there is a sequence $g_n$ of functions in $X$ such that
	\begin{align*}
		\lVert g_n \rVert_X &\leq 1, \\
		\lVert T(g_n) \rVert_Y &\geq n(2C)^{n+1}.
	\end{align*}
	By putting $f_n = \lvert g_n \rvert$ we obtain a sequence of non-negative functions $f_n$ that satisfy
	\begin{align*}
		\lVert f_n \rVert_X &\leq 1, \\
		\lVert T(f_n) \rVert_Y &\geq n(2C)^{n+1}\frac{1}{c_1},
	\end{align*}
	where the second estimate holds because $T$ satisfies the condition \eqref{T1}. It follows that
	\begin{equation*}
		\sum_{n=0}^{\infty} C^{n+1}\lVert (2C)^{-n-1}f_n \rVert_X \leq \sum_{n=0}^{\infty} 2^{-n-1} < \infty,
	\end{equation*}
	and thus Theorem~\ref{TC} implies that $f = \sum_{n=0}^{\infty} (2C)^{-n-1}f_n \in X$. Note that this sum converges to $f$ both in $X$ and $\mu$-a.e., which together with the non-negativeness of the functions $f_n$ yields that it holds for every $k \in \mathbb{N}$ that
	\begin{align*}
		f &= \sum_{n=0}^{\infty} (2C)^{-n-1}f_n \geq \sum_{n=0}^{k} (2C)^{-n-1}f_n \geq (2C)^{-k-1}f_k & \text{$\mu$-a.e.}
	\end{align*}
	Now, $T$ satisfies the condition \eqref{T2} and it therefore holds that
	\begin{equation*}
		\lVert T(f) \rVert_Y \geq \frac{(2C)^{-k-1}}{c_2} \lVert T(f_k) \rVert_Y \geq \frac{k}{c_2 c_1}
	\end{equation*}
	for all $k \in \mathbb{N}$. Hence, $T(f) \notin Y$, which establishes the sufficiency.
\end{proof}

An important special case of Theorem~\ref{TOQBFS} is an extremely useful result that tells us that an embedding between two quasi-Banach function spaces is always continuous.

\begin{corollary} \label{CEQBFS}
	Let $(X,\lVert \cdot \rVert_X)$ and $(Y,\lVert \cdot \rVert_Y)$ be quasi-Banach function spaces. If $X \subseteq Y$ then also $X \hookrightarrow Y$.
\end{corollary}

As another consequence of Theorem~\ref{TC} we present here the following result concerning associate spaces. Statements of this type are sometimes called Landau's resonance theorems and they can be quite useful as a tool in the study of associate spaces. For the more classical version concerning Banach function spaces see for example \cite[Chapter~1, Lemma~2.6]{BennettSharpley88}.

\begin{theorem} \label{LT}
	Let $(X,\lVert \cdot \rVert_X)$ be a quasi-Banach function space and let $\lVert \cdot \rVert_{X'}$ and $X'$, respectively, be the corresponding associate norm and associate space. Then arbitrary function $f \in M$ belongs to $X'$ if and only if it satisfies
	\begin{equation} \label{LT1}
		\int_R \lvert f g \rvert \: d\mu < \infty
	\end{equation}
	for all $g \in X$.
\end{theorem}

\begin{proof}
	The necessity is an immediate consequence of the Hölder inequality (Theorem~\ref{THAS}).
	
	As for the sufficiency, denote by $C$ the modulus of concavity of $\lVert \cdot \rVert_X$ and suppose that $f \notin X'$. By the definition of $X'$, this means that there exists some sequence $g_n$ of non-negative functions in $X$ such that $\lVert g_n \rVert_X \leq 1$ while
	\begin{equation*}
		\int_R \lvert f g_n \rvert \: d\mu > n (2C)^{n+1}.
	\end{equation*}
	Then, as in the preceding theorem, we obtain that
	\begin{equation*}
		\sum_{n=0}^{\infty} C^{n+1} \lVert (2C)^{-n-1} g_n \rVert_X < \infty
	\end{equation*}
	which yields us, by the means of Theorem~\ref{TC}, a function $g = \sum_{n=0}^{\infty} (2C)^{-n-1} g_n \in X$ which satisfies $g \geq (2C)^{-n-1} g_n$ $\mu$-a.e.~for all $n \in \mathbb{N}$ and thus
	\begin{equation*}
		\int_R \lvert fg \rvert \: d\mu \geq (2C)^{-n-1} \int_R \lvert f g_n \rvert \: d\mu \geq n
	\end{equation*}
	for all $n \in \mathbb{N}$. That is, we have shown that there is a $g \in X$ which violates \eqref{LT1}.
\end{proof}

The final application we present here is a result about those quasi-Banach function norms that do not satisfy the axiom \ref{P5} of Banach function spaces.

\begin{theorem} \label{TP5}
	Let $(X,\lVert \cdot \rVert_X)$ be a quasi-Banach function space. Suppose that $E \subseteq R$ is a set such that $\mu(E) < \infty$ and that for every constant $K \in (0, \infty)$ there is a non-negative function $f \in X$ satisfying
	\begin{equation*}
		\int_E \lvert f \rvert \: d\mu > K \lVert f \rVert_X.
	\end{equation*}
	Then there is a non-negative function $f_E \in X$ such that
	\begin{equation} \label{TP5.1}
		\int_E f_E \: d\mu = \infty.
	\end{equation}
\end{theorem}

\begin{proof}
	Denote by $C$ the modulus of concavity of $\lVert \cdot \rVert_X$ and consider some sequence $f_n$ of non-negative functions in $X$ such that
	\begin{align*}
		\lVert f_n \rVert_X \leq 1, \\
		\int_E f_n \: d\mu > n (2C)^{n+1}.
	\end{align*}
	As before, we get that
	\begin{equation*}
		\sum_{n=0}^{\infty} C^{n+1} \lVert (2C)^{-n-1} f_n \rVert_X < \infty
	\end{equation*}
	and we thus obtain a function $f_E = \sum_{n=0}^{\infty} (2C)^{-n-1} f_n \in X$ satisfying $f_E \geq (2C)^{-n-1} f_n$ $\mu$-a.e.~for all $n \in \mathbb{N}$. Consequently,
	\begin{equation*}
		\int_E f_E \: d\mu > n
	\end{equation*}
	for all $n \in \mathbb{N}$ which shows that $f_E$ satisfies \eqref{TP5.1}.
\end{proof}

\subsection{Separability} \label{SS}

In this section we examine the separability of quasi-Banach function spaces. We restrict ourselves to the case when $(R, \mu) = (\mathbb{R}^n, \lambda^n)$ and arrive to the expected conclusion that a quasi-Banach function space is separable if and only if it has absolutely continuous quasinorm. Note that our approach differs from that which is usually used in the context of Banach function spaces, see e.g.~\cite[Chapter~1, Section~5]{BennettSharpley88}. This is due to the fact that this classical approach depends on the fact that the Banach function spaces have separating dual and is therefore unusable in the context quasi-Banach function spaces which, in general, do not have this property.

We first introduce some auxiliary terms as well as some notation.

\begin{definition}
	Let $k\in \mathbb{Z}$ and let $a=(a_1,a_2,\dots,a_n)\in \mathbb{Z}^n$. Denote by $Q_{k,a}$ the dyadic cube
	\begin{equation*}
	Q_{k,a}=\prod_{i=1}^n\Big(\frac{a_i}{2^k},\frac{a_i+1}{2^k}\Big).
	\end{equation*}
	Moreover, denote by $\mathcal{D}_k$ the collection
	\begin{equation*}
	\mathcal{D}_k=\{Q_{k,a}; \; a\in \mathbb{Z}^n\}
	\end{equation*}
	of all dyadic cubes of order $k$ and by $\mathcal{D}$ the collection
	\begin{equation*}
	\mathcal{D}=\bigcup_{k\in \mathbb{Z}} \mathcal{D}_k
	\end{equation*}
	of all dyadic cubes.
	
	We say that a set $\Omega\subset \mathbb{R}^n$ is a complex of order $k$ if there are finitely many sets $Q_i \in \mathcal{D}_k$ that satisfy
	\begin{align*}
	&\Omega=\bigcup_i Q_i.
	\end{align*}
\end{definition}

Note that $\mathcal{D}$ is a countable collection of sets.

\begin{definition}
	We denote by $\mathcal{S}$ the following family of simple functions:
	\begin{equation*}
	\mathcal{S}=\left \{ f \in M(\mathbb{R}^n, \lambda^n); \; f=\sum_{i=1}^k \alpha_i \chi_{Q_i}, Q_i\in \mathcal{D}, \alpha_i\in \mathbb{C}, \textup{Re}(\alpha_i) \in \mathbb{Q}, \textup{Im}(\alpha_i) \in \mathbb{Q} \right \}.
	\end{equation*}
\end{definition}

Note that $\mathcal{S}$ is a countable family of functions.

The last thing we need in order to prove our results is the following covering lemma.

\begin{lemma} \label{LC}
	Let $K\subset \mathbb{R}^n$ be a compact set. Then for any open set $G$ such that $K \subseteq G$, any $\varepsilon >0$ and any $k_0 \in \mathbb{Z}$ there is a complex $\Omega$ of order $k$, where $k \geq k_0$, that has the following properties:
	\begin{enumerate}
		\item \label{LC1}
		$\Omega \subseteq G$,
		\item \label{LC2}
		$\lambda^n (K \setminus \Omega) = 0$,
		\item \label{LC3}
		$\lambda^n(\Omega \setminus K) < \varepsilon$,
		\item  \label{LC4}
		if $Q \in \mathcal{D}_{k}$ satisfies $Q \subseteq \Omega$ then $Q \cap K \neq \emptyset$.
	\end{enumerate}
\end{lemma}

\begin{proof}
	Find some open set $\tilde{H}$ such that $K \subseteq \tilde{H}$ and $\lambda^n (\tilde{H} \setminus K) < \varepsilon$, set $H = G \cap \tilde{H}$ and put $\delta = \textup{dist}(K, \mathbb{R}^n \setminus H) > 0$. Then there is a $k \geq k_0$ such that $2^{-k} \sqrt{n} < \delta$. Put
	\begin{equation*}
	\Omega = \bigcup \left \{ Q \in \mathcal{D}_k; \; Q \cap K \neq \emptyset \right \}.
	\end{equation*}
	Then $\Omega$ is a complex of order $k$ (because $K$ is bounded) and clearly has the properties \eqref{LC2} and \eqref{LC4}.  Furthermore, it holds for any $x \in \Omega$ that it belongs to some $Q_x \in \mathcal{D}_k$ such that $Q_x \cap K \neq \emptyset$ and thus
	\begin{equation*}
	\textup{dist}(x, K) \leq \textup{diam}(Q_x) = 2^{-k} \sqrt{n} < \delta.
	\end{equation*}
	Hence, $\Omega \subseteq H$ and the remaining properties \eqref{LC1} and \eqref{LC3} follow.
\end{proof}

We are now suitably equipped to prove our results. We begin by the sufficiency in its following precise form.

\begin{theorem} \label{TAAC}
	Let $(X, \lVert \cdot \rVert_X)$ be a quasi-Banach function space over $M(\mathbb{R}^n, \lambda^n)$. Suppose that $\chi_K$ has absolutely continuous quasinorm for any compact $K \subseteq \mathbb{R}^n$. Then for any function $f \in X$ that has absolutely continuous quasinorm and any $\varepsilon > 0$ there is a function $s \in \mathcal{S}$ such that $\lVert f - s \rVert_X < \varepsilon$.
\end{theorem}

\begin{proof}
	Denote by $C$ the modulus of concavity of $\lVert \cdot \rVert_X$ and fix some $f \in X$ and $\varepsilon > 0$. Because $f$ is finite $\lambda^n$-a.e. by Theorem~\ref{TEiMF} and is assumed to have absolutely continuous quasinorm, there is some $N > 0$ such that the restrictions of $f$ onto the sets
	\begin{align*}
	E_0 &= \left \{ x \in \mathbb{R}^n; \; \lvert f(x) \rvert > N \right \},	\\
	E_1 &= \left \{ x \in \mathbb{R}^n; \; \lvert x \rvert > N \right \}
	\end{align*}
	satisfy
	\begin{align}
	\lVert f \chi_{E_0} \rVert_X &< \varepsilon, \label{TAAC1}	\\	
	\lVert f \chi_{E_1} \rVert_X &< \varepsilon. \label{TAAC2}
	\end{align}
	Consider now the set $E = \mathbb{R}^n \setminus (E_0 \cup E_1)$ and denote $L = \lVert \chi_E \rVert_X \in (0, \infty)$. We may assume $L>0$, because otherwise $\chi_E = 0$ $\lambda^n$-a.e.~in which case $f$ can be approximated by the zero function. Since $E$ is bounded we can find some compact set $\tilde{K}$ such that there is an open set $G$ satisfying $E \subseteq G \subseteq \tilde{K}$. Thanks to the assumption that both $f$ and $\chi_{\tilde{K}}$ have absolutely continuous quasinorm we may find $\delta > 0$ such that it holds for all $A \subseteq G$ satisfying $\lambda^n(A) < \delta$ that both
	\begin{align}
	\lVert f \chi_A \rVert_X &< \varepsilon, \label{TAAC3} \\
	\lVert \chi_A \rVert_X &< \frac{\varepsilon}{N + \frac{\varepsilon}{L}}. \label{TAAC4}
	\end{align}
	By the classical Luzin theorem (see e.g.~\cite[Theorem~2.24]{Rudin87}) there is a compact $K \subseteq E$ such that $\lambda^n(E \setminus K) < \delta$ and that the restriction of $f$ onto $K$ is uniformly continuous on $K$. We may thus find some $\Delta > 0$ such that it holds for all $x, y \in K$ satisfying $\lvert x - y \rvert < \Delta$ that
	\begin{equation}
	\lvert f(x) - f(y) \rvert < \frac{\varepsilon}{L}. \label{TAAC5}
	\end{equation}
	
	Find now some $k_0 \in \mathbb{N}$ such that $2^{-k_0}\sqrt{n} < \Delta$. By Lemma~\ref{LC} there is a complex $\Omega$ of order $k$, where $k \geq k_0$, such that $\Omega \subseteq G$, $\lambda^n(K \setminus \Omega) = 0$ and $\lambda^n(\Omega \setminus K) < \delta$. Denote by $Q_i$ the finite sequence of dyadic cubes from $\mathcal{D}_k$ for which $\Omega = \bigcup_i Q_i$. Choose for every $i$ some arbitrary point $x_i \in Q_i \cap K$ (existence of this point follows from Lemma~\ref{LC}) and find some $a_i \in \{z \in \mathbb{C}; \; \textup{Re}(z) \in \mathbb{Q}, \textup{Im}(z) \in \mathbb{Q} \}$ such that
	\begin{equation}
	\lvert a_i - f(x_i)\rvert < \frac{\varepsilon}{L}. \label{TAAC6}
	\end{equation}
	We are now in position to define $s$ by
	\begin{align*}
	s(x) &= 
	\begin{cases}
	0, &x \notin \Omega, \\
	a_i, &x \in Q_i.
	\end{cases}
	\end{align*}
	
	It remains to estimate $\lVert f - s \rVert_X$. We expand it to get that
	\begin{equation*}
	\lVert f -s \rVert_X \leq C  \lVert (f -s)\chi_K \rVert_X + C \lVert (f -s)\chi_{\mathbb{R}^n \setminus K} \rVert_X
	\end{equation*}
	and estimate the two terms on the right-hand side separately.
	
	Consider first the term $\lVert (f -s)\chi_K \rVert_X$. Because $\chi_K = \chi_{K \cap \Omega}$ $\lambda^n$-a.e.~it suffices to estimate $\lVert (f -s)\chi_{K \cap \Omega} \rVert_X$. To this end, consider arbitrary point $x \in K \cap \Omega$ and find the appropriate index $i$ such that $x \in Q_i$. Then
	\begin{equation*}
	\lvert x - x_i \rvert \leq \textup{diam}(Q_i) = 2^{-k}\sqrt{n} < \Delta
	\end{equation*}
	and thus, by \eqref{TAAC5} and \eqref{TAAC6},
	\begin{equation*}
	\lvert f(x) - s(x) \rvert \leq \lvert f(x) - f(x_i) \rvert + \lvert f(x_i) - s(x) \rvert < \frac{2 \varepsilon}{L}.
	\end{equation*}
	From this uniform estimate it now follows that
	\begin{equation*}
	\lVert (f -s)\chi_K \rVert_X = \lVert (f -s)\chi_{K \cap \Omega} \rVert_X \leq \frac{2 \varepsilon}{L} \lVert \chi_{K \cap \Omega} \rVert_X \leq \frac{2 \varepsilon}{L} \lVert \chi_E \rVert_X = 2 \varepsilon.
	\end{equation*}
	
	As for the remaining term $\lVert (f -s)\chi_{\mathbb{R}^n \setminus K} \rVert_X$, because $\chi_{E_0} + \chi_{E_1} + \chi_{E \setminus K} \geq \chi_{\mathbb{R}^n \setminus K}$  $\lambda^n$-a.e.~we may further expand it by Lemma~\ref{NT} to get
	\begin{equation*}
	\begin{split}
	\lVert (f -s)\chi_{\mathbb{R}^n \setminus K} \rVert_X &\leq C \lVert f \chi_{\mathbb{R}^n \setminus K} \rVert_X + C \lVert s \chi_{\mathbb{R}^n \setminus K} \rVert_X \leq \\
	&\leq C^4 \lVert f \chi_{E_0} \rVert_X + C^3 \lVert f \chi_{E_1} \rVert_X + C^2 \lVert f \chi_{E \setminus K} \rVert_X + C \lVert s \chi_{\mathbb{R}^n \setminus K} \rVert_X
	\end{split}
	\end{equation*}
	and again examine those four terms separately.
	
	The terms $\lVert f \chi_{E_0} \rVert_X$, $\lVert f \chi_{E_1} \rVert_X$ and $\lVert f \chi_{E \setminus K} \rVert_X$ are estimated immediately by \eqref{TAAC1}, \eqref{TAAC2} and \eqref{TAAC3}, respectively, one only has to remember in the last case that $E \setminus K \subseteq G$ and $\lambda^n(E \setminus K) < \delta$.
	
	Finally we turn ourselves to the term $\lVert s \chi_{\mathbb{R}^n \setminus K} \rVert_X$. The function $s$ is zero on $\mathbb{R}^n \setminus \Omega$ and also bounded on $\mathbb{R}^n$ by $N + \frac{\varepsilon}{L}$, while the set $\Omega \setminus K$ satisfies that both $\lambda^n(\Omega \setminus K) < \delta$ and $\Omega \setminus K \subseteq G$. Hence, we get by \eqref{TAAC4} that
	\begin{equation*}
	\lVert s \chi_{\mathbb{R}^n \setminus K} \rVert_X = \lVert s \chi_{\Omega \setminus K} \rVert_X \leq \left ( N + \frac{\varepsilon}{L} \right ) \lVert \chi_{\Omega \setminus K} \rVert_X < \left ( N + \frac{\varepsilon}{L} \right ) \frac{\varepsilon}{N + \frac{\varepsilon}{L}} = \varepsilon.
	\end{equation*}
	
	By combining the above obtained estimates we arrive at the conclusion that
	\begin{equation*}
	\lVert f -s \rVert_X \leq (2C + C^2 + C^3 + C^4 + C^5) \varepsilon.
	\end{equation*}
\end{proof}

\begin{theorem} \label{TS}
	Let $(X,\lVert \cdot \rVert_X)$ be a quasi-Banach function space over $M(\mathbb{R}^n, \lambda^n)$. Then $X$ is separable if and only if it has absolutely continuous quasinorm.
\end{theorem}

\begin{proof}
	The sufficiency follows directly from Theorem~\ref{TAAC}.
	
	As for the necessity, assume that there is a non-negative function $f \in X$, $\varepsilon > 0$ and a sequence $E_k$ of subsets of $\mathbb{R}^n$ such that $\chi_{E_k} \downarrow 0$ $\lambda^n$-a.e.~while
	\begin{equation*}
	\lVert f \chi_{E_k} \rVert_X > \varepsilon.
	\end{equation*}
	We will now construct, by induction, a strictly increasing sequence of natural numbers $k_i$ such that the formula
	\begin{equation} \label{TS0}
	f_i = f \left ( \chi_{E_{k_{i}}} - \chi_{E_{k_{i+1}}} \right )
	\end{equation} 
	will define a sequence of functions $f_i$ such that it will hold for every $i$ that
	\begin{enumerate}[series = Prop]
		\item \label{TS1}
		$\lVert f_i \rVert_X > \varepsilon$,
		\item \label{TS2}
		$\textup{supp}(f_i) \cap \textup{supp}(f_j) = \emptyset$ for all $j \in \mathbb{N}$ such that $j \neq i$,
		\item \label{TS3}
		$f_i \leq f$ $\lambda^n$-a.e.
	\end{enumerate}
	
	The construction proceeds as follows. Set $k_0 = 0$. Assume now that we already have $k_i$ for some $i \geq 0$ and consider the sequence $f (\chi_{E_{k_i}} - \chi_{E_k} )$, $k \geq k_i$. Since $\chi_{E_k} \downarrow 0$, we get that $f (\chi_{E_{k_i}} - \chi_{E_k} ) \uparrow f \chi_{E_{k_i}}$and thus
	\begin{equation*}
	\lVert f (\chi_{E_{k_i}} - \chi_{E_k} ) \rVert_X \uparrow \lVert f \chi_{E_{k_i}} \rVert_X > \varepsilon.
	\end{equation*}
	Consequently, there is an index $k_{i+1} > k_i$ such that the function $f_{i+1}$, defined by \eqref{TS0}, satisfies the requirement \eqref{TS1}. Since the condition \eqref{TS3} is satisfied trivially, it remains only to consider the condition \eqref{TS2}.To this end, note that $\textup{supp}(f_{i+1}) \cap E_{k_{i+1}} = \emptyset$	which gives for any $j < i+1$ that
	\begin{equation*}
	\textup{supp}(f_{i+1}) \cap \textup{supp}(f_{j}) = \emptyset.
	\end{equation*}
	This concludes the construction.
	
	Now, consider the following family of functions:
	\begin{equation*}
	\mathcal{F} = \left \{ f_{\eta} \in M(\mathbb{R}^n, \lambda^n); \; \eta \in \{0, 1\}^{\mathbb{N}}, f_{\eta} = \sum_{\substack{j \in \mathbb{N}, \\ \eta(j) = 1}} f_j \right \}.
	\end{equation*}
	Thanks to the properties \eqref{TS2} and \eqref{TS3} of the functions $f_i$, we obtain that it holds for every $\eta \in \{0, 1\}^{\mathbb{N}}$ that $f_{\eta} \leq f$ and thus $\mathcal{F} \subseteq X$. Furthermore, if $\eta_1, \eta_2 \in \{0, 1\}^{\mathbb{N}}$ satisfy for some $j_0$ that $\eta_1(j_0) \neq \eta_2(j_0)$ then we have $\lvert f_{\eta_1} - f_{\eta_2} \rvert \geq f_{j_0}$ $\lambda^n$-a.e.~and consequently
	\begin{equation*}
	\lVert f_{\eta_1} - f_{\eta_2} \rVert_X \geq \lVert f_{j_0} \rVert_X > \varepsilon.
	\end{equation*}
	It follows that $\mathcal{F}$ is an uncountable discrete subset of $X$ and $X$ therefore cannot be separable.
\end{proof}

\subsection{Boundedness of the dilation operator} \label{SBDO}

The final question we are dealing with in this paper is whether the dilation operator is bounded on r.i.~quasi-Banach function spaces over $(\mathbb{R}^n, \lambda^n)$. As stated in Theorem~\ref{TBD}, the answer is positive, but an entirely new method had to be developed, since the proof of the classical result from the setting of r.i.~Banach function spaces uses either interpolation or the property that $X$ coincides with its second associate space $X''$, neither of which is in general available in the setting of quasi-Banach function spaces.

The first step is the rather obvious observation that the dilation operator $D_a$ is, in the following sense, monotone with respect to the parameter $a$.

\begin{lemma} \label{LMDO}
	Let $(X,\lVert \cdot \rVert_X)$ be an r.i.~quasi-Banach function space over $M(\mathbb{R}^n, \lambda^n)$. Assume that $0 < a < b < \infty$. Then the dilation operators $D_a$ and $D_b$ satisfy
	\begin{equation*}
		\lVert D_b f \rVert_X \leq \lVert D_a f \rVert_X
	\end{equation*}
	for all $f \in M(\mathbb{R}^n, \lambda^n)$. Consequently,
	\begin{enumerate}
		\item
		if $a \in (0, \infty)$, $D_a$ is bounded and $b > a$ then $D_b$ is also bounded,
		\item
		if $b \in (0, \infty)$, $D_b$ is unbounded and $a < b$ then $D_a$ is also unbounded.
	\end{enumerate}
\end{lemma}

\begin{proof}
	It follows from the monotonicity of $f^{\star}$ with respect to $\lvert x \rvert$ and the assumption $a < b$ that it holds for all $f \in M(\mathbb{R}^n, \lambda^n)$ and all $x \in \mathbb{R}^n$ that
	\begin{equation*}
		f^{\star}(bx) \leq f^{\star}(ax).
	\end{equation*}
	Because $\lVert \cdot \rVert_X$ is assumed to be rearrangement invariant the conclusion holds by the following chain of equalities and inequalities:
	\begin{equation*}
		\lVert D_bf \rVert_X = \lVert (D_bf)^{\star} \rVert_X = \lVert D_b(f^{\star}) \rVert_X \leq \lVert D_a(f^{\star}) \rVert_X = \lVert (D_af)^{\star} \rVert_X = \lVert D_af \rVert_X.
	\end{equation*}
\end{proof}

The next lemma shows that in order to prove the boundedness of $D_a$ for any $a \in (0, \infty)$ it is actually sufficient to prove it for any fixed $a \in (0, 1)$.

\begin{lemma} \label{LSup}
	Let $(X,\lVert \cdot \rVert_X)$ be an r.i.~quasi-Banach function space over $M(\mathbb{R}^n, \lambda^n)$. Set
	\begin{equation*}
		s = \sup\{a \in [0, \infty); \; \lVert D_a \rVert = \infty\},
	\end{equation*}
	where $\lVert D_a \rVert$ is the operator norm of $D_a$. Then $s$ is equal to either zero or one.
\end{lemma}

\begin{proof}
	Because $D_1$ is the identity operator on $M(\mathbb{R}^n, \lambda^n)$, which is bounded on $X$, it follows from Lemma~\ref{LMDO} that $s \in [0, 1]$.
	
	Assume $s \in (0, 1)$ and choose some $t \in (s, \sqrt{s})$. Then $t^2 < s < t$ and therefore we get from Lemma~\ref{LMDO} and the definition of $s$ that $D_t$ is bounded while $D_{t^2}$ is unbounded. But this is a contradiction, since $D_{t^2} = D_t \circ D_t$ and thus it must be bounded with a norm that is at most equal to $\lVert D_t \rVert^2$. 
\end{proof}

We now proceed by proving the boundedness of $D_b$ for $b = \left ( \frac{2}{3} \right )^{\frac{1}{n}}$. To this end, we will employ the four auxiliary operators defined bellow.

\begin{definition}
	Set
	\begin{align*}
		G_1 & = \bigcup_{k \in \mathbb{Z}} \left (\frac{1}{2^{2k+2}},\frac{1}{2^{2k+1}}\right), \\
		G_2 & = \bigcup_{k \in \mathbb{Z}} \left (\frac{1}{2^{2k+1}},\frac{1}{2^{2k}}\right).
	\end{align*}
	We then define the operators $R_1$, $R_2$ for functions in $f \in M((0, \infty), \lambda)$ by
	\begin{align*}
		R_1f &= f \chi_{G_1}, \\
		R_2f &= f \chi_{G_2}
	\end{align*}
	and the corresponding operators $S_1$ and $S_2$ for functions $f \in M(\mathbb{R}^n, \lambda^n)$ by
	\begin{align*}
		S_1f(x) &= [R_1f^*](\alpha_n \lvert x \rvert^n), \\
		S_2f(x) &= [R_2f^*](\alpha_n \lvert x \rvert^n)
	\end{align*}
	for all $x \in \mathbb{R}^n$, where $\alpha_n$ is the same constant as in the Definition~\ref{DRNIR} of radial non-increasing rearrangement.
\end{definition}

\begin{lemma} \label{L1D}
	Let $g \in M((0, \infty), \lambda)$. Then it holds for $i=1,2$ and almost every $t \in (0, \infty)$ that
	\begin{align*}
		&(R_ig^*)^*(t) \leq D_{\frac{3}{2}}g^*(t).
	\end{align*}
\end{lemma}
\begin{proof} 
	We perform the proof only for $i=1$ because the other case is analogous.
	
	Denote $g_1 = R_1 g^*$ and fix some $k \in \mathbb{Z}$ and some $x \in \left ( \frac{1}{2^{2k+2}}, \frac{1}{2^{2k+1}} \right )$. Then the non-increasing rearrangement shifts the value $g_1(x)$ to a point $t_x$ given by
	\begin{equation*}
		t_x = x -\frac{1}{2^{2k+2}} + \sum_{l=k+1}^\infty \left ( \frac{1}{2^{2l+1}} - \frac{1}{2^{2l+2}} \right ) = x - \frac{2}{3}\frac{1}{2^{2k+2}}.
	\end{equation*}
	It is easy to see that this $t_x$ satisfies $t_x \in \left (\frac{2}{3}\frac{1}{2^{2(k+1)+1}}, \frac{2}{3}\frac{1}{2^{k+1}} \right)$. Conversely, given any $t$ in this interval, one can find an appropriate $x_t \in \left ( \frac{1}{2^{2k+2}}, \frac{1}{2^{2k+1}} \right )$ such that $t = t_{x_t}$. Note that intervals of this form cover almost every point of $(0, \infty)$. 
	
	Fix now arbitrary $t \in (0, \infty)$ such that there is some $k \in \mathbb{Z}$ satisfying $t\in \left (\frac{2}{3}\frac{1}{2^{2(k+1)+1}}, \frac{2}{3}\frac{1}{2^{2k+1}} \right )$. We have by the arguments presented above that
	\begin{equation*}
		g_1^*(t)=g^*(x_t),
	\end{equation*}
	where $x_t$ satisfies
	\begin{equation*}
		x_t=t+\frac{2}{3}\frac{1}{2^{2k+2}}=t+\frac{1}{2}\frac{2}{3}\frac{1}{2^{2k+1}}\ge t+\frac{1}{2}t=\frac{3}{2}t.
	\end{equation*}
	Since $g^*$ is non-increasing we have 
	\begin{equation*}
		g_1^*(t)=g^*(x_t)\leq D_{\frac{3}{2}}g^*(t)
	\end{equation*}
	for almost every $t \in (0, \infty)$.
\end{proof}

\begin{lemma} \label{LB}
	Let $(X,\lVert \cdot \rVert_X)$ be an r.i.~quasi-Banach function space over $M(\mathbb{R}^n, \lambda^n)$ and denote its modulus of concavity by $C$. Put $b = \left ( \frac{2}{3} \right )^{\frac{1}{n}}$. Then the dilation operator $D_b$ is bounded on $X$ and its operator norm satisfies $\lVert D_b \rVert \leq 2C$.
\end{lemma}

\begin{proof}
	The operators $S_1$ and $S_2$ were defined in such a way that they satisfy, for any $f \in M(\mathbb{R}^n, \lambda^n)$, the following two crucial properties:
	\begin{align}
		f^{\star} &= S_1 f + S_2 f & \textup{$\lambda^n$-a.e.~on $\mathbb{R}^n$,} \label{LB1}	\\
		(S_i f)^* &= (R_i f^*)^* & \textup{on $\mathbb{R}$ for $i = 1, 2$.} \label{LB2}
	\end{align}
	While \eqref{LB1} is obvious directly from the respective definitions, the property \eqref{LB2} is more involved, as it follows from examination of the level sets $\left \{ S_if > t  \right \}$ and $\left \{ R_i f^* > t  \right \}$ that uses both their structure and the fact that the preimage of an interval in $(0, \infty)$ with respect to the mapping $x \mapsto \alpha_n \lvert x \rvert^n$ is an annulus centered at zero whose $n$-dimensional measure is equal to the length of the original interval. We leave out the details.
	
	Using Lemma~\ref{L1D} (with $g=f^*$) in combination with \eqref{LB2} we now obtain for $i = 1, 2$ and almost every $x \in \mathbb{R}^n$ the following:
	\begin{equation} \label{LB3}
	(S_i f)^{\star} (x) = (S_if)^*(\alpha_n \lvert x \rvert^n) = (R_i f^*)^*(\alpha_n \lvert x \rvert^n) \leq (D_{\frac{3}{2}}f^*)(\alpha_n \lvert x \rvert^n) = D_{b^{-1}} f^{\star}(x).
	\end{equation}
	
	We are now suitably equipped to prove the boundedness of $D_b$, since we may obtain from \eqref{LB1} and \eqref{LB3} that it holds for all $f \in M(\mathbb{R}^n, \lambda^n)$ that
	\begin{equation*}
		\begin{split}
			\lVert D_b f \rVert_X &= \lVert (D_b f)^{\star} \rVert_X = \\
			&= \lVert S_1(D_b f) + S_2(D_b f) \rVert_X \leq \\
			&\leq C (\lVert S_1(D_b f) \rVert_X + \lVert S_2(D_b f) \rVert_X) = \\
			&= C (\lVert (S_1(D_b f))^{\star} \rVert_X + \lVert (S_2(D_b f))^{\star} \rVert_X) \leq \\
			&\leq C (\lVert D_{b^{-1}}(D_b f)^{\star} \rVert_X + \lVert D_{b^{-1}}(D_b f)^{\star} \rVert_X) = \\
			&= 2C \lVert f^{\star}  \rVert_X = \\
			&= 2C \lVert f \rVert_X.
		\end{split}
	\end{equation*}
\end{proof}

The desired boundedness of $D_a$ for any $a \in (0, \infty)$ now follows by combining the previous results. Moreover, we also obtain an upper bound on the operator norm of $D_a$.

\begin{theorem} \label{TBD}
	Let $(X,\lVert \cdot \rVert_X)$ be an r.i.~quasi-Banach function space over $M(\mathbb{R}^n, \lambda^n)$. Denote its modulus of concavity by $C$. Then the dilation operator $D_a$ is bounded on $X$ for any $a \in (0, \infty)$  and its operator norm satisfies
	\begin{align*}
		\lVert D_a &\rVert \leq
		\begin{cases}
			2C a^{\frac{\log(2C)}{\log(b)}} & \textup{for $a \in (0,1)$}, \\
			1 &\textup{for $a \in [1, \infty)$},
		\end{cases}
	\end{align*}
	where $b$ is as in Lemma~\ref{LB}, that is, $b = \left (\frac{2}{3} \right)^{\frac{1}{n}}$.
\end{theorem}

\begin{proof}
	The boundedness follows directly from Lemma~\ref{LB} and Lemma~\ref{LSup} while the estimate on the norm in case $a \in [1, \infty)$ follows from Lemma~\ref{LMDO} and the fact that $D_1$ is the identity operator on $M(\mathbb{R}^n, \lambda^n)$. It thus remains only to prove the estimate for $a \in (0,1)$.
	
	Let $a \in (0,1)$ and fix $k \in \mathbb{N}$ such that
	\begin{equation*}
		a \in \left [ b^{k+1},  b^k \right ).
	\end{equation*}
	It then follows from Lemma~\ref{LMDO} and Lemma~\ref{LB} that
	\begin{equation*}
		\lVert D_a \rVert \leq \lVert D_{b^{k+1}} \rVert \leq \lVert D_b \rVert^{k+1} \leq (2C)^{k+1} \leq 2C a^{\frac{\log(2C)}{\log(b)}}.
	\end{equation*}
\end{proof}

\bibliographystyle{abbrv}
\bibliography{bibliography}

\end{document}